\documentclass[11pt]{article}   
\usepackage{mathptmx}      
\usepackage[lined,boxed]{algorithm2e}
\usepackage{amsmath, amssymb}
\usepackage{color}
\usepackage{multirow}
\usepackage{bm}
\usepackage{graphicx}
\usepackage{tikz}
\usepgflibrary{shapes.misc}
\usetikzlibrary{%
   arrows,%
   calc,%
   fit,%
   patterns,%
   plotmarks,%
   shapes.geometric,%
   shapes.misc,%
   shapes.symbols,%
   shapes.arrows,%
   shapes.callouts,%
   shapes.multipart,%
   shapes.gates.logic.US,%
   shapes.gates.logic.IEC,%
   er,%
   automata,%
   backgrounds,%
   chains,%
   topaths,%
   trees,%
   petri,%
   mindmap,%
   matrix,%
   calendar,%
   folding,%
   fadings,%
   through,%
   positioning,%
   scopes,%
   decorations.shapes,%
   decorations.text,%
   decorations.pathmorphing,%
   decorations.pathreplacing,%
   decorations.footprints,%
   shadows}
\tikzset{
terminal/.style={
rectangle,minimum size=6mm,rounded corners=1mm,
very thick,draw=black!50,
top color=white,bottom color=black!20,
font=\ttfamily},
inicio/.style={
circle,minimum size=.3mm,
very thick, draw=red!50!black!50,
top color=white,bottom color=black!20,
font=\ttfamily}
}

\usepackage{amsthm}

\newtheorem{theorem}{Theorem}

\begin{document}

\title{Exact cost minimization of a series-parallel reliable system with multiple component choices using an algebraic method
\thanks{ This paper has been partially supported by Junta de Andaluc\'{\i}a under grant FQM-5849, and Ministerio de Ciencia e Innovaci\'{o}n MTM2010-19336, MTM2010-19576 and FEDER.}
}

\title{Exact cost minimization of a series-parallel system}        

\author{F. Castro, J. Gago, I. Hartillo, J. Puerto, J.M. Ucha
}




\maketitle

\thanks{{\em{In memory of our fellow Alejandro Fern\'{a}ndez-Margarit}}

\date{}

\begin{abstract}
The redundancy allocation problem is formulated minimizing the design cost for a series-parallel system with multiple component choices whereas ensuring a given system reliability level. The obtained model is a nonlinear integer programming problem with a non linear, non separable constraint. We propose an algebraic method, based on Gr\"obner bases, to obtain the exact solution of the problem. In addition, we provide a closed form for the required Gr\"obner bases, avoiding the bottleneck associated with the computation, and promising computational results.
\end{abstract}

\section{Introduction}
\label{intro}
System reliability is considered an important measure in the engineering design process. A series system is like a chain composed of links,
each of them representing a subsystem. The failure of one of these components means the failure of the whole system. In order to avoid this, it is usual to use redundant components in parallel to guarantee a certain level of reliability. These systems are called series-parallel systems.

Determining the optimal number of components in each subsystem is the so called reliability optimization problem. Two different approaches are usual:
\begin{itemize}
  \item maximize system reliability subject to system budget constraint, or
  \item minimize system cost subject to a required level of reliability.
\end{itemize}
Both problems are nonlinear integer programming problems, and they are NP-hard \cite{chern}. There are very few papers looking for their exact solutions, due to the difficulty of the problems. Those works use essentially Dynamic Programming \cite{dynamic}, branch and bound methods \cite{Djerdjour}, or Lagrangian relaxation \cite{ruan}, among others techniques. On the contrary, in the literature there are many heuristics and metaheuristic algorithms; such as those
based in Genetic Algorithms \cite{coit}, Tabu Search \cite{tabu} or Ant Colony Optimization \cite{ant}, among others.

In this paper we study the exact solution of one of the versions of the problem that minimizes the cost function of the chosen design, subject to a non linear constraint which describes the reliability of the considered system. For a fixed subsystem its inner components can be considered equal, as in \cite{Djerdjour}, or different, as in \cite{ruan}. If the components are equal the reliability function is separable and convex, and the problem can be reduced to a linear knapsack problem \cite{Djerdjour}. In the case of multiple component choices the reliability function is no longer separable. In \cite{dynamic}, the solution is found using dynamic programming methods. That approach presents two stages; in the first one the problem is restricted to each subsystem, with a level of reliability. Under this assumption the reliability function is separable, and the optimization problem can be reduced to a knapsack problem. Then the reliability levels of the subsystems are determined by a new dynamic programming process.

The solution method shown in \cite{ruan} uses an algorithm based on Lagrangian relaxations over two linear relaxations of the original problem. The first relaxation consists of deleting the non-linear reliability function, and adding certain linear constraints, one for each subsystem. The second relaxation assumes that the same type of component is going to be used in every subsystem, so that the problem has the form as in \cite{Djerdjour}.


Mainly, the algorithm of \cite{ruan} is a what their authors called a {\em cut and partition scheme} (a geometric branch and bound). The solution space is partitioned in boxes, which are divided and discarded for certain conditions.
The cuts are built from the best bound feasible solution of some Lagrangian relaxations. Such bounds allow to remove certain boxes depending on the improvement with respect to the current best point.

We address the problem via a different approach based on Gr\"obner bases. As introduction on this subject, we recommend the text books \cite{adams}, \cite{bertsimaslibro} and \cite{cox}.

Gr\"{o}bner bases were applied to Integer Linear Programming, by the first time, in \cite{conti}. Later, Tayur et al. \cite{tayur} introduced a new application framework, which solves nonlinear integer programming problems, with a linear objective function. This is exactly our framework, as in \cite{CGHPU}.

First, we consider a relaxed integer programming problem where all the restrictions are linear. Then we find the solution of the relaxed problem by computing a test set. By using the so called reverse test set, we can solve the complete problem, generating paths from the solution of the relaxed problem to a solution of the complete one. These paths increase the cost function at each step.

A test set for a linear integer programming problem is a set of directions that can be used to design descending algorithms with respect to a linear cost function. A test set can be computed from a Gr\"{o}bner basis of the toric ideal associated with the linear restrictions, with respect to an order given by the cost function.

One of the main tasks in the process described before is usually the calculation of the Gr\"obner basis. We construct a linear programming problem from the original one, removing the reliability function, and adding a new linear restriction. This constraint is obtained computing a feasible solution with a greedy algorithm. For the relaxed linear programming problem obtained in this way we explicitly give the associated Gr\"obner basis, and so the test set to solve the main problem. We point out here that any Gr\"{o}bner basis computation is done using closed formulas, thus avoiding the hard computation burden of reduction algorithms to compute Gr\"{o}bner bases.

The organization of the paper is as follows. In Section \ref{sec:model} we introduce the notation and describe the model of a parallel-series system with multiple component choices. In Section \ref{sec:greedy}, a greedy algorithm is described to compute a feasible point. Section \ref{sec:review-gr} is devoted to a brief introduction to the essential facts about Gr\"{o}bner bases. Section \ref{sec:teorema} contains the main result about the closed formula for the test set of the integer linear problem. Our computational experiments are reported in Section \ref{sec:computational}. Finally we draw some concluding remarks in Section \ref{sec:conclusiones}.

\section{General model} \label{sec:model}
In order to formulate the problem, some notation is first introduced.
\begin{itemize}
 \item $n$ number of subsystems.
 \item $k_i$ number of different types of available components for the $i$-th subsystem, $i=1,\ldots, n$.
 \item $r_{ij}$ reliability of the $j$-th component for the $i$-th subsystem, $i=1,\ldots,n$, $j=1,\ldots,k_i$.
 \item $c_{ij}$ cost of the $j$-th component for the $i$-th subsystem, $i=1,\ldots,n$, $j=1,\ldots,k_i$.
 \item $l_{ij}, u_{ij}$ lower/upper bounds of number of $j$ components for the $i$-th subsystem, $i=1,\ldots,n$, $j=1,\ldots,k_i$.
 \item $R_0$ admissible level of reliability of the whole system.
 \item $x_{ij}$ number of $j$ components used in the $i$-th subsystem, $i=1,\ldots,n$, $j=1,\ldots,k_i$.
\end{itemize}

In our model, some assumptions are considered:
\begin{itemize}
 \item Components have two states: working or failed.
 \item The reliability of each component is known and is deterministic.
 \item Failure of individual components are independent.
 \item Failed components do not damage other components or the system, and they are not repaired.
\end{itemize}

\begin{figure}

\begin{tikzpicture}[point/.style={coordinate},>=stealth',thick,draw=black!50,
caja invisible/.style={rectangle,minimum size=2pt,font=\ttfamily},
tip/.style={-,shorten >=1pt},every join/.style={rounded corners},
hv path/.style={to path={-| (\tikztotarget)}},
vh path/.style={to path={|- (\tikztotarget)}},
node distance=5mm and 5mm]

\matrix[row sep=1mm,column sep=3mm,ampersand replacement=\&] {
\& \node (p1) [point] {}; \&  \node (r1) [terminal]{$r_{1 1}$}; \& \node (p6) [point] {}; \&
\node (p11) [point] {}; \&  \node (r6) [terminal]{$r_{2 1}$}; \& \node (p16) [point] {}; \& \&
\node (p21) [point] {}; \&  \node (r11) [terminal]{$r_{n 1}$}; \& \node (p26) [point] {};\&
\\
\& \node (p2) [point] {};\& \node (r2) [terminal]{$r_{1 2}$}; \& \node (p7) [point] {};\&
\node (p12) [point] {}; \&  \node (r7) [terminal]{$r_{2 2}$}; \& \node (p17) [point] {};\& \&
\node (p22) [point] {}; \&  \node (r12) [terminal]{$r_{n 2}$}; \& \node (p27) [point] {};\&
\\
\node (in) [inicio] {}; \& \node (p3) [point] {};\& \node (c1) [caja invisible]{$\vdots$}; \&
\node (p8) [point] {}; \&
\node (p13) [point] {}; \&  \node (c3) [caja invisible]{$\vdots$}; \& \node (p18) [point] {}; \&
\node (c5) [caja invisible]{$\cdots$}; \&
\node (p23) [point] {}; \&  \node (c6) [caja invisible]{$\vdots$}; \& \node (p28) [point] {};\&
\node (out) [inicio] {};
\\
\& \node (p4) [point] {};\& \node (c2) [caja invisible]{}; \& \node (p9) [point] {};\&
\node (p14) [point] {}; \&  \node (c4) [caja invisible]{}; \& \node (p19) [point] {};\& \&
\node (p24) [point] {}; \&  \node (c7) [caja invisible]{}; \& \node (p29) [point] {};\&
\\
\& \node (p5) [point] {};\& \node (r5) [terminal]{$r_{1 k_1}$};  \& \node (p10) [point] {}; \&
\node (p15) [point] {}; \&  \node (r10) [terminal]{$r_{2 k_2}$}; \& \node (p20) [point] {};\& \&
\node (p25) [point] {}; \&  \node (r15) [terminal]{$r_{n k_n}$}; \& \node (p30) [point] {};\&
\\
};

{ [start chain]
\chainin (in);
\chainin (p3) [join];
{ [start branch=p3]
\chainin (p3) [join];
\chainin (p2) [join];
\chainin (r1) [join=by vh path];
\chainin (p7) [join=by hv path];
\chainin (r2) [join];
\chainin (p2) [join];
}

\chainin (p3) [join];
\chainin (p4) [join];
\chainin (r5) [join=by vh path];
\chainin (p9) [join=by hv path];
\chainin (p8) [join];
{ [start branch=p8]
\chainin (p8) [join];
\chainin (p7) [join];
}
{ [start branch=p8]
\chainin (p8) [join];
\chainin (p9) [join];
}
\chainin(p13) [join];
{ [start branch=p13]
\chainin (p13) [join];
\chainin (p12) [join];
\chainin (r6) [join=by vh path];
\chainin (p17) [join=by hv path];
\chainin (r7) [join];
\chainin (p12) [join];
}
\chainin (p13) [join];
\chainin (p14) [join];
\chainin (r10) [join=by vh path];
\chainin (p19) [join=by hv path];
\chainin (p18) [join];
{ [start branch=p18]
\chainin (p18) [join];
\chainin (p17) [join];
}
\chainin (c5) [join];
\chainin (p23) [join];
{ [start branch=p23]
\chainin (p23) [join];
\chainin (p22) [join];
\chainin (r11) [join=by vh path];
\chainin (p27) [join=by hv path];
\chainin (r12) [join];
\chainin (p22) [join];
}
\chainin (p23) [join];
\chainin (p24) [join];
\chainin (r15) [join=by vh path];
\chainin (p29) [join=by hv path];
\chainin (p28) [join];
{ [start branch=p28]
\chainin (p28) [join];
\chainin (p27) [join];
}
\chainin (out)[join];
};

\end{tikzpicture}
\caption{A series-parallel system with multiple choice components}
\label{sistema1}
\end{figure}
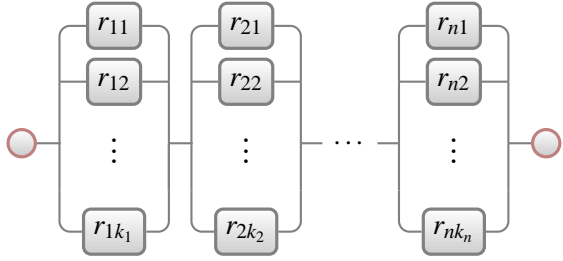

This model is illustrated in Figure \ref{sistema1}. It is a system with $n$ subsystems with the notation introduced before. The optimization problem can be formulated as:
\begin{equation}
\begin{array}{lcll}
  (RP)   & \min & \sum_{i=1}^n \sum_{j=1}^{k_i} c_{i j} x_{i j} \\
  & \mbox{s. t.}        & \\
        & & R(x) \geq R_0,  \\ \noalign{\medskip}
         & & \sum_{j=1}^{k_i} x_{ij}\geq 1, & i=1,\ldots,n, \label{eq:almenos} \\ \noalign{\medskip}
       & & 0 \leq l_{ij}\leq x_{i j} \leq u_{ij}, & i=1,\ldots,n,  \\
        & &  & j=1,\ldots,k_i, \\
        & & x_{ij} \in {\mathbb Z}_+ & \mbox{for all } i, j,
\end{array}
\end{equation}
where $R(x) = \prod_{i=1}^n (1-\prod_{j=1}^{k_i}(1-r_{i j})^{x_{i j}})$.
The first $n$ linear inequalities in (\ref{eq:almenos}) assert that each subsystem must have, at least, one component.

As usual, we can make a change of variables $y_{ij}=x_{ij}-l_{ij}$, so that we can assume $l_{ij} = 0$. This does not alter the equations of $(RP)$, and some of the last equations can be redundant. Hence it can be assumed $l_{ij}=0$ without loss of generality, and:
$$
\begin{array}{lcll}
 (RP) &  \min & \sum_{i=1}^n \sum_{j=1}^{k_i} c_{i j} x_{i j}\\
 & \mbox{s. t.} & \\
      &       & R(x) \geq R_0, \\ \noalign{\medskip}
      &       & \sum_{j=1}^{k_i} x_{ij}\geq 1,& i=1,\ldots,n. \\ \noalign{\medskip}
      &       & 0\leq x_{i j} \leq u_{ij}, & i=1,\ldots,n, \\ & & & j=1,\ldots,k_i, \\
      &       & x_{ij} \in {\mathbb Z}_+ & \mbox{for all } i, j.
\end{array}
$$
A feasible solution is sometimes called a reliable solution because it ensures a reliability greater than or equal to $R_0$.

\section{Computing a reliable system with a greedy procedure} \label{sec:greedy}

The main step used in the algebraic algorithm described in this article is to consider an integer linear programming problem $(LRP)$, relaxed from the original problem $(RP)$. There is only one nonlinear constraint in $(RP)$, the equation which ensures the reliability of the whole system. Removing the nonlinear constraint we get an integer linear programming problem:
$$
\begin{array}{lcll}
 (LRP1) &  \min & \sum_{i=1}^n \sum_{j=1}^{k_i} c_{i j} x_{i j}\\
 & \mbox{s. t. } & \\
        &     & \sum_{j=1}^{k_i} x_{ij}\geq 1, & i=1,\ldots,n. \\ \noalign{\medskip}
        &     & 0\leq x_{i j} \leq u_{ij}, & i=1,\ldots,n, \\ & & & j=1,\ldots,k_i, \\
         & & x_{ij} \in {\mathbb Z}_+ & \mbox{for all } i, j.
\end{array}
$$
In our solution technique, we start from the solution of the linear programming problem $(LRP1)$, and following the directions given by the test set of $(LRP1)$ we follow a descent path to the solution of the complete problem $(RP)$. If the linear relaxation is too weak, the paths to be followed to get to the optimal solution of $(RP)$ will be very long, and the number of points to be processed is huge.

To avoid this problem we add a new linear equation. We need a feasible point $y^0$ of $(RP)$, and there are lots of heuristic methods to obtain such a point. In our case, we use a greedy algorithm similar to \cite{Djerdjour} or \cite{ruan}.

At the beginning of the greedy algorithm, $y^0$ describes the system with the maximum number of components of every type. If the reliability of that system is less than $R_0$, the problem is unfeasible. We consider $I$ the set of all pairs $(i,j)$, which describes the $j$-th component for the $i$-th subsystem. For each $(i,j)$, we calculate the rate $t_{ij} = \frac{c_{ij}}{-\log(1-r_{ij})}$ between cost and reliability, and order $I$ non increasingly by these rates (ties are solved by lex order, for example). For the first index $(i_0,j_0)$ in $I$, we subtract components of type $(i_0,j_0)$ from $y^0$ until it is non reliable, there is no such component or the $i_0$-th subsystem is empty. If the solution obtained by this process is non reliable, or the $i_0$-th subsystem is empty, one $(i_0,j_0)$ component is added. Then we take the next index in the set $I$ and repeat the procedure, until the index set $I$ has been completely processed.

\begin{algorithm}
  \DontPrintSemicolon
  \SetAlgoLined
  \KwData{$r_{ij}$, vector $\bm{c}$}
  \KwResult{$\bm{y}^0$ feasible point}
  $\bm{y}^0=(u_{11},\ldots,u_{nk_n})$ \;
  $t=\left(\frac{c_{11}}{-\log(1-r_{11})},\ldots,\frac{c_{nk_n}}{-\log(1-r_{nk_n})}\right)$ \;
  $I=\{(1,1),\ldots,(1,k_1)\ldots,(n,k_n)\}$ \;
  Order $I$ non increasingly by $t_{i,j}$ \;
 \ForAll{$(i,j)\in I$}
 {
   Reliable=TRUE \;
   SubsystemNonEmpty=TRUE \;
  \While{Reliable {\bf and} SubsystemNonEmpty {\bf and} $y^0_{i,j}>0$  }{
     $y^0_{i,j}=y^0_{i,j}-1$ \;
     \If{$\sum_k y^0_{ik}<1$}
      {
      SubsystemNonEmpty=FALSE \;
      }
     \If{$R(y^0)<R_0$} {
        Reliable=FALSE \;
        }
     \If{Reliable=FALSE {\bf or} SubsystemNonEmpty=FALSE}{
       $y^0_{ij}=y^0_{ij}+1$ \;
       }
   }
  }
\caption{Greedy algorithm}
\label{greedy}
\end{algorithm}

Using the above greedy algorithm, we obtain a feasible point $y^0$, with a cost $\sum_{ij}c_{ij}y^0_{ij}=c^0$. The optimal solution of $(RP)$ has a cost less than or equal to $c^0$, so we can add to $(RP)$ a valid inequality stating this condition and the problem has an equivalent form:
$$
\begin{array}{lcll}
 (RP) &  \min  & \sum_{i=1}^n \sum_{j=1}^{k_i} c_{i j} x_{i j}\\
& \mbox{s. t. } \\
& & R(x) \geq R_0, \\ \noalign{\medskip}
& & \sum_{j=1}^{k_i} x_{ij}\geq 1, & i=1,\ldots,n, \\ \noalign{\medskip}
& & 0\leq x_{i j} \leq u_{ij}, & i=1,\ldots,n, \\ & & & j=1,\ldots,k_i, \\ \noalign{\medskip}
& &\sum_{i=1}^n \sum_{j=1}^{k_i} c_{ij}x_{ij} \leq c^0, \\
        & & x_{ij} \in {\mathbb Z}_+ & \mbox{for all } i, j.
\end{array}
$$
From this formulation, we have the new integer linear problem
$$
\begin{array}{lcll}
 (LRP) &  \min & \sum_{i=1}^n \sum_{j=1}^{k_i} c_{i j} x_{i j} \\
& \mbox{s. t. } & \\
& &  \sum_{j=1}^{k_i} x_{ij} \geq 1, & i=1,\ldots,n, \\ \noalign{\medskip}
& & 0 \leq x_{i j} \leq u_{ij}, & i=1,\ldots,n, \\ & & & j=1,\ldots,k_i, \\ \noalign{\medskip}
& & \sum_{i=1}^n \sum_{j=1}^{k_i} c_{ij} x_{ij} \leq c^0, \\
        & & x_{ij} \in {\mathbb Z}_+ & \mbox{for all } i, j.
\end{array}
$$

\section{A review on integer programming and Gr\"obner bases} \label{sec:review-gr}

In this section, we recall the concepts and algorithms used to solve Integer Linear Programming problems from an algebraic point of view, and the walk back procedure for nonlinear integer programming problems based on test sets. To this end, we have followed \cite{sturmfels} and \cite{tayur}.

\subsection{Gr\"{o}bner bases}

Denote by $k[\bm{x}] = k[x_1,\ldots, x_N]$ the ring of polynomial with coefficients in a field $k$. In our case, $k$ will be ${\mathbb R}$. The ideal generated by a subset ${\mathcal F} \subset k[\bm{x}]$ is the set $\langle {\mathcal F} \rangle$ consisting of all linear combinations:
$$
\langle {\mathcal F} \rangle = \{ h_1 f_1 + \cdots + h_r f_r ~:~ f_1, \ldots, f_r \in {\mathcal F}, h_1, \ldots, h_r \in k[\bm{x}] \}.
$$

A term order on ${\mathbb N}^N$ is a total order $\prec$ satisfying the following properties:
\begin{itemize}
\item $\prec$ is compatible with sums, i.e., $\alpha \prec \beta\Rightarrow \alpha+\gamma \prec \beta+\gamma$, for all $\alpha,\beta,\gamma\in{\mathbb N}^N$.
\item $\prec$ is a well-ordering, i.e., $0 \prec \alpha$ for all $\alpha\in{\mathbb N}^N$, $\alpha\neq 0$.
\end{itemize}

If we fix a term order $\prec$, then every non zero polynomial $f$ has a unique initial term ${\rm in}_{\prec}(f) = a \bm{x}^{\alpha}$. It is the monomial $a \bm{x}^{\alpha}$ where ${\alpha}$ is the largest term appearing in $f$ for the term order $\prec$. We are particularly interested in two term orders:
\begin{enumerate}
  \item The lexicographic order $<_{\rm lex}$. For every $\alpha, \beta \in {\mathbb N}^N$, we say $\alpha >_{\rm lex} \beta$ if, in the vector difference $\alpha - \beta \in {\mathbb Z}^N$, the leftmost nonzero entry is positive.
  \item The vector induced order $<_{\bm{c}}$. We consider a vector $\bm{c} \in {\mathbb N}^N$. Given $\alpha, \beta \in {\mathbb N}^N$, we say $\alpha >_{\bm{c}} \beta$ if
      $$
      \bm{c}^t \alpha > \bm{c}^t \beta \mbox{ or } \bm{c}^t \alpha = \bm{c}^t \beta, \mbox{ and } \alpha >_{\rm lex} \beta.
      $$
\end{enumerate}
For example, consider the polinomial $f = 6x_1 x_2^2 x_3 + 7x_3^2 - 5x_1^3 + 4x_1^2 x_3^2$ and the vector $\bm{c} = (3,2,2)^t$. Then
$$
{\rm in}_{<_{\rm lex}}(f) = -5x_1^3, {\rm in}_{<_{\bm{c}}} (f) = 4x_1^2 x_3^2.
$$
Of course, we can reorder the variables $x_i$, and get a new term order. In general, the notation $<_{\bm{c}}$ means a term order which respect the partial order defined by the vector $\bm{c}$ and then a tie-break term order, so if we change the lexicographic order in the definition on $<_{\bm{c}}$, we get another vector induced order.

Suppose that $J$ is an ideal in $k[\bm{x}]$, and $\prec$ is a given term order. Then its initial ideal is the ideal generated by the initial terms of the polynomials in $J$:
$$
{\rm in}_{\prec}(J) = \langle {\rm in}_{\prec} (f) ~:~ f\in J \rangle.
$$
A finite subset ${\mathcal G}$ of $J$ is a Gr\"{o}bner basis with respect to the term order $\prec$ if the initial terms of the elements in ${\mathcal G}$ suffice to generate the initial ideal:
$$
{\rm in}_{\prec}(J) = \langle {\rm in}_{\prec} (g) ~:~ g \in {\mathcal G} \rangle.
$$
Fixed an ideal and a term order, a Gr\"obner basis is not unique. Adding two more conditions, the uniqueness is guaranteed. The reduced Gr\"obner basis of $J$ with respect to $\prec$ is a Gr\"obner basis ${\mathcal G}_{\prec}$ of $J$ such that:
\begin{itemize}
\item ${\rm in}_{\prec}(g_i)$ has unit coefficient for each $g_i \in {\mathcal G}_{\prec}$.
\item For each $g_i \in {\mathcal G}_{\prec}$, no monomial in $g_i$ lies in $\langle {\rm in}_{\prec}({\mathcal G}_{\prec} \backslash g_i) \rangle$.
\end{itemize}
Every ideal $J$ has a unique reduced Gr\"{o}bner basis for each term order.

\subsection{Test set}
Consider a linear programming problem:
$$
\begin{array}{lcll}
LP(\bm{b}) & \min & \bm{c}^t \cdot \bm{x}\\
& \mbox{s. t. }&  \\
& & A\cdot \bm{x}=\bm{b}, \\
& & \bm{x}\in {\mathbb Z}_+^N, \\
\end{array}
$$
where $ A\in{\mathbb Z}^{d\times N}, \bm{b} \in{\mathbb Z}^d, \bm{c}\in{\mathbb R}^N$. The notation $LP(\bm{b})$ denotes the linear programming problem with the right-hand-side restrictions fixed to $\bm{b}$. When we write $(LP)$, we note all integer programming problems, obtained by varying the right-hand-side vector $\bm{b}$, fixing $A$ and the cost function $\bm{c}$. Consider the map $\pi: {\mathbb N}^N \to {\mathbb Z}^d$ defined by $\pi(\bm{x}) = A \bm{x}$. Given a vector $\bm{b} \in {\mathbb Z}^d$, the set $\pi^{-1}(\bm{b}) = \{ \bm{u} \in {\mathbb N}^N ~:~ \pi(\bm{u}) = \bm{b} \}$ is the fiber of $(LP)$ over $\bm{b}$.

We group points in ${\mathbb N}^N$ according to increasing cost value $\bm{c}^t \bm{x}$, and refine this order to a total order $<_{\bm{c}}$ breaking ties among points with the same cost value by adopting some term order (lexicographic, for example, as defined in the previous section). It is the vector induced order.

A set $G_{<_{\bm{c}}} \subset {\mathbb Z}^N$ is a test set for the family of integer problems $(LP)$ with respect to the matrix $A$ and the order $<_{\bm{c}}$ if
\begin{itemize}
\item for each nonoptimal point $\alpha$ in each fiber of $(LP)$, there exists $g\in G_{<_{\bm{c}}}$ such that $\alpha-g$ is a feasible solution in the same fiber and $\alpha-g<_{\bm{c}}\alpha$,
\item for the optimal point $\beta$ in a fiber of $(LP)$, $\beta-g$ is unfeasible for every $g\in G_{<_{\bm{c}}}$
\end{itemize}
A test set for $(LP)$ gives an obvious algorithm to solve an integer program, provided we know a feasible solution to this problem. At every step of this algorithm, we have two different cases:
\begin{itemize}
  \item There exists an element in the test set which, when subtracted from the current point, yields an improved point. We are then in a nonoptimal point, but we get a better one.
  \item There will not exist such an element in the set, so we are in the optimum of the fiber.
\end{itemize}

\subsection{Toric ideal}
We define $I_A$ the toric ideal associated with $A$ as
$$
I_A = \langle \bm{x}^{\alpha} - \bm{x}^{\beta} ~:~ A \alpha = A \beta , \alpha, \beta \in {\mathbb N}^N \rangle.
$$
Given an integral vector $\gamma \in {\mathbb Z}^N$, we can write it uniquely as $\gamma = \gamma^+ - \gamma^-$, where $\gamma^+, \gamma^- \in {\mathbb N}^N$ and have disjoint supports. It is well known (\cite{sturmfels}) that
$$
I_A = \langle \bm{x}^{\alpha^+} - \bm{x}^{\alpha^-} ~:~ A \alpha = \bm{0}, \alpha \in {\mathbb Z}^N \rangle.
$$

The relationship between the previous concepts is that the reduced Gr\"obner basis ${\mathcal G}_{<_{\bm{c}}}$ of $I_A$ with respect to the order $<_{\bm{c}}$ allows us to compute a uniquely defined minimal test set $G_{<_{\bm{c}}}$ for $(LP)$. The reduced Gr\"{o}bner basis is formed by binomials
$$
{\mathcal G}_{<_{\bm{c}}} = \{ \bm{x}^{\alpha_i} - \bm{x}^{\beta_i}, i=1,2,\ldots,r \}, \mbox{ with }{\rm in}_{<_{\bm{c}}} ( \bm{x}^{\alpha_i} - \bm{x}^{\beta_i}) = \bm{x}^{\alpha_i},
$$
and then the test set is expressed as
$$
G_{<_{\bm{c}}} = \{ \alpha_i - \beta_i, i=1,2,\ldots, r \}.
$$

\subsection{Walk back procedure}
Basically the walk back procedure gives an algorithm which computes the optimum for a nonlinear integer programming problem under some conditions. The integer programming problem $(RP)$ introduced in Section \ref{sec:model} is not linear. It has a nonlinear constraint (the reliability condition), while the rest of the restrictions are linear and the cost function is also linear. These are the conditions required to use the walk back procedure, introduced in \cite{tayur}. In Algorithm \ref{walkbackalgorithm}, it is used the directed graph defined by the Gr\"obner basis over the feasible points, but directions are reversed in the skeleton. In each step, elements $w=\alpha+g$ in the reverse skeleton are computed, where $A \alpha = \bm{0}$ and $g$ is an element in the Gr\"{o}bner basis.

In general, Algorithm \ref{walkbackalgorithm}, uses the following notation. We denote by $(RP)$ the entire non linear integer programming problem, $(LRP)$ the relaxed linear integer programming problem which arises from $(RP)$. Let $\beta$ be the optimum of $(LRP)$. If $\beta$ is feasible for $(RP)$, then it is the solution to $(RP)$. If it is non feasible, then the reverse skeleton is needed.

Let $P(\alpha)$ denote the path, in the directed graph (reversed) of the linear integer programming Problem $(LRP)$, from the optimum $\beta$ for $(LRP)$ to a feasible point $\alpha$ for $(LP)$. There is always one. Any solution of $(RP)$ is feasible for $(LRP)$, so the objective is to find such a path, in an ordered way. In each reversed step the cost function increases, so the minimum cost feasible points for $(RP)$ are found first.

\begin{algorithm}
\DontPrintSemicolon
  \SetAlgoLined
  \SetKwFunction{groebner}{groebner}
  \SetKwFunction{optimo}{OptimumLP}
  \SetKwFunction{greedy}{greedy}
  \KwData{Matrix $A$, vectors $\bm{b}, \bm{c}$, non-linear restrictions}
  \KwResult{Optimum}
  ${\mathcal P} = \{ P(\beta) \}$ \;
  $y^0 = \greedy(RP)$ \;
  $Y = \{ y^0 \}$\;
  ${\mathcal G}_{<_{\bm{c}}} = \groebner(I_A)$ with respect to $<_{\bm{c}}$ \;
  $\beta = $ optimum for relaxed $LRP$ \;
  \Repeat{all paths in ${\mathcal P}$ are pruned}
  {
  \ForAll{$P(\alpha) \in {\mathcal P}$}
  {
  \ForAll{$g \in {\mathcal G}_{<_{\bm{c}}}$}
  {
  $w = \alpha + g$ \;
  \If{$w$ is a feasible point of (LRP)}
  {
    \eIf{$w$ is feasible for (RP)}
    {
     $Y = Y \cup \{w\} $\;
     Prune $P(w)$ \;
    }
    {
      \If{$y <_{\bm{c}} w$ for some $y \in Y$}
      {
       Prune $P(w)$\;
      }
    ${\mathcal P} = {\mathcal P} \cup \{ P(w) \}$\;
    }
   }
  }
  }
  Delete $P(\alpha)$ from ${\mathcal P}$ \;
  }
  $Optimum$ = Select minimum $<_{\bm{c}}$ element from $Y$\;
\caption{Walk back procedure}
\label{walkbackalgorithm}
\end{algorithm}

\section{The test set for the relaxed linear problem} \label{sec:teorema}

Once the $(LRP)$ problem is reinforced by means of the linear constraint that comes after a feasible solution of $(RP)$ is found by the greedy algorithm, the relaxed linear problem is:

$$
\begin{array}{lcll}
 (LRP) &  \min & \sum_{i=1}^n \sum_{j=1}^{k_i} c_{i j} x_{i j}\\
& \mbox{ s. t. } & \\
& & \sum_{j=1}^{k_i} x_{ij}\geq 1, & i=1,\ldots,n, \\ \noalign{\medskip}
& & 0\leq x_{i j} \leq u_{ij}, & i=1,\ldots,n, \\ & & & j=1,\ldots,k_i, \\ \noalign{\medskip}
& &\sum_{i=1}^n \sum_{j=1}^{k_i} c_{ij}x_{ij} \leq c^0, \\
& & x_{ij} \in {\mathbb Z}_+ & \mbox{for all } i,j.
\end{array}
$$

Each inequality must be converted to an equality, so we must introduce a new slack variable for each inequality:

$$
\begin{array}{lcll}
 (LRP) &  \min & \sum_{i=1}^n \sum_{j=1}^{k_i} c_{i j} x_{i j} \\
& \mbox{s.t.} & \\
& &  \sum_{j=1}^{k_i} x_{ij} - d_i = 1, & i=1,\ldots,n, \\ \noalign{\medskip}
& & x_{i j} + t_{ij}= u_{ij}, & i=1,\ldots,n, \\ & & & j=1,\ldots,k_i, \\
& & \sum_{i=1}^n \sum_{j=1}^{k_i} c_{ij} x_{ij} + b= c^0, \\
& & x_{ij} \in {\mathbb Z}_+ & \mbox{for all } i,j.
\end{array}
$$
If we put $N = k_1+\ldots+k_n$ and
$$D_{n \times N} =
    \begin{pmatrix}
      \overbrace{1  \ldots  1}^{k_1} & \overbrace{0 \ldots  0}^{k_2} & \ldots & \overbrace{0  \ldots  0}^{k_n} \\
       0  \ldots 0 & 1  \ldots  1 & \ldots & 0  \ldots  0 \\
      && \ddots \\
      0  \ldots  0 & 0  \ldots  0 & \ldots & 1  \ldots  1 \\
    \end{pmatrix}
 $$
 the restrictions in matrix form can be written as
$$
 \begin{pmatrix}
  D & -I_n & {\bf 0}_{n \times N} & {\bf 0}_{n \times 1} \\
  I_N & {\bf 0}_{N \times n} & I_N & {\bf 0}_{N \times 1} \\
  {\bf c}_{1 \times N} & {\bf 0}_{1 \times n} & {\bf 0}_{1 \times N} & 1\\
 \end{pmatrix}
 \cdot
 \begin{pmatrix}
  \bm{x}_{N \times 1} \\
  \bm{d}_{n \times 1} \\
  \bm{t}_{N \times 1} \\
  b \\
 \end{pmatrix}=
\begin{pmatrix}
  {\bf 1}_n \\
  \bm{u}_1 \\ \vdots \\ \bm{u}_n \\ c^0
 \end{pmatrix},
$$
 where
 $$
 \begin{array}{l}
 {\bf c}_{1 \times N} = \left( \begin{array}{ccccccc} c_{11} & \ldots & c_{1 k_1} & \ldots & c_{n1} & \ldots & c_{n k_n} \end{array} \right), \\
 \bm{x} = \left( \begin{array}{c} x_{11} \\ \vdots \\ x_{1k_1} \\ \vdots \\ x_{n1} \\ \vdots \\ x_{nk_n} \end{array} \right), \bm{t} = \left( \begin{array}{c} t_{11} \\ \vdots \\ t_{1k_1} \\ \vdots \\ t_{n1} \\ \vdots \\ t_{nk_n} \end{array} \right),
 \bm{d} = \left( \begin{array}{c} d_1 \\ \vdots \\ d_n \end{array} \right), \bm{u}_i = \left( \begin{array}{c} u_{i1} \\ \vdots \\ u_{ik_i} \end{array} \right), i=1,\ldots,n,
 \end{array}
 $$
 and ${\bf 1}_n$ denotes the $n$-vector with all the componentes equal to $1$. We can assume that, for each $i=1,\ldots,n$, the costs $c_{ij}$ are ordered in descending order: $c_{iq} \ge c_{ip}$ if $q < p$.
 Let
 $$
 \bm{z} = (x_{11}, \ldots, x_{nk_n}, d_1, \ldots, d_n, t_{11}, \ldots, t_{nk_n}, b) = (\bm{x}, \bm{d}, \bm{t}, b),
 $$
 and consider the following set of binomials in $k[\bm{z}]$:
 $$
 {\mathcal G} = \{ \underline{x_{ik} d_i} - t_{ik} b^{c_{ik}}, \underline{x_{iq} t_{ip}} - x_{ip} t_{iq} b^{c_{iq} - c_{ip}} \},
 $$
   for $i=1,\ldots,n, k=1,\ldots, k_i, 1 \le q < p \le k_i$. Let $>$ be a term order in $k[\bm{z}]$ such that $\bm{x} > \bm{d} > \bm{t} > b$. Within each block, the variables are sorted lexicographically as follows:
   $$
   x_{11} > \cdots > x_{1k_1} > x_{21} > \cdots > x_{nk_n}, t_{11} > \cdots > t_{1k_1} > t_{21} > \cdots > t_{nk_n}, d_1 > \cdots > d_n.
   $$
 \begin{theorem}\label{thm:base}
   The set ${\mathcal G}$ is the reduced Gr\"{o}bner basis of the toric ideal $I_A$ with respect to the term orden $>$. Moreover, ${\mathcal G}$ is the reduced Gr\"{o}bner basis with respect to the order $<_{\bm{c}}$ induced by the cost vector ${\bf c}$.
 \end{theorem}
 \begin{proof}
   The proof follows similar steps and notation that \cite[Thm. 4]{tayur}. First of all, the set ${\mathcal G}$ is a subset of $I_A$, because all the binomials $\bm{z}^{\alpha} - \bm{z}^{\beta}$ in ${\mathcal G}$ verify $A \alpha = A \beta$.

   The initial term of every binomial in ${\mathcal G}$ with respect to $>$ is the underlined term. It is enough to show that for every binomial $\bm{z}^{\alpha} - \bm{z}^{\beta} \in I_A$, with initial term $\bm{z}^{\alpha}$, there is some $g \in {\mathcal G}$ whose initial term divides $\bm{z}^{\alpha}$. By definition of toric ideal, $\bm{z}^{\alpha} - \bm{z}^{\beta} \in I_A$ if and only if $\alpha - \beta \in K = \{ \bm{y} \in {\mathbb Z}^s ~:~ A \bm{y} = {\bf 0} \}, s= n+2N+1$. We denote an element $\bm{y}$ in $K$ by $\bm{y} = (y_{x}, y_d, y_t, y_b)$ to indicate the correspondence between components of $\bm{y}$ and the columns of $A$. In addition, we denote the components of $y_x$ by $(X_{11}, \ldots, X_{nk_n})$, and similarly for the others. We classify the elements in $K$ in the following manner:
   \begin{enumerate}
     \item Let $K_1 = \{ \bm{y} \in K ~:~ y_x = 0 \}$. Now $\bm{y} \in K_1$ if and only if $(y_d, y_t, y_b) \in {\mathbb Z}^{s_1}, s_1 = n+N+1$ belongs to the lattice $S' = \{ \bm{w} \in {\mathbb Z}^{s_1} ~:~ A' \bm{w} = 0 \}$ where
         $$
         A' = \left( \begin{array}{ccc} -I_n \\ & I_N \\ & & 1 \end{array} \right).
         $$
         But $S' = 0$ since $A'$ is a non singular matrix. Therefore $K_1 = 0$. This implies that there are no binomials of the form $\bm{z}^{\alpha} - \bm{z}^{\beta}$ that do not contain the variables $x_{ij}$.
     \item Let $K_2 = \{ \bm{y} \in K ~:~ y_d = 0 \}$. Again $\bm{y} \in K_2$ if and only if $(y_x, y_t, y_b) \in {\mathbb Z}^{s_2}, s_2 = 2N+1$ belongs to the lattice $S'' = \{ \bm{w} \in {\mathbb Z}^{s_2} ~:~ A'' \bm{w} = 0 \}$, where
         $$
         A'' = \left( \begin{array}{ccc} D & 0 & 0 \\ I_N & I_N & 0 \\ {\bf c} & 0 & 1 \end{array} \right).
         $$
         Let $X_{iq}$ be the left most nonzero component of $y_x$. We may assume that $X_{iq} > 0$ since $S$ is the set of integer points in a vector space which implies that it contains the negative of every element in it. The $i$-th row of matrix $D$ in $A''$ implies that there exists some $p > q$ such that $X_{ip} < 0$. Therefore,
         $$
         x_{iq} \mbox{ divides } \bm{z}^{\bm{y}^+} \mbox{ and } x_{ip} \mbox{ divides } \bm{z}^{\bm{y}^-}.
         $$
         Consider now the rows given by the block $\left( \begin{array}{ccc} I_N & I_N & 0 \end{array} \right)$ in $A''$. These rows imply that $T_{iq} = - X_{iq} < 0$ and $T_{ip} = - X_{ip} > 0$. Therefore,
         $$
         x_{iq} t_{ip} \mbox{ divides } \bm{z}^{\bm{y}^+} \mbox{ and } x_{ip} t_{iq} \mbox{ divides } \bm{z}^{\bm{y}^-}.
         $$
         The initial term of $\bm{z}^{\bm{y}^+} - \bm{z}^{\bm{y}^-}$ with respect to $>$ is $\bm{z}^{\bm{y}^+}$ since $x_{iq}$ divides $\bm{z}^{\bm{y}^+}$ and $x_{iq}$ is the greatest variable that appears in this binomial. But this implies that the initial term of $x_{iq} t_{ip} - x_{ip} t_{iq} b^{c_{iq} - c_{ip}} \in {\mathcal G}$ divides the initial term of $\bm{z}^{\bm{y}^+} - \bm{z}^{\bm{y}^-}$. Therefore, the initial term of all binomials associated with $K_2$ is divisible by the initial term of an element in ${\mathcal G}$.
     \item Consider now a general element in $S = \{ \bm{y} \in {\mathbb Z}^s ~:~ A \bm{y} = 0 \}, s = n+2N+1$, with no variables restricted to be zero. By the previous cases we may assume $y_x \ne 0, y_d \ne 0$. Let $D_i$ be the first nonzero component of $y_d$. As before, we may assume that $D_i > 0$. Therefore,
         $$
         d_i \mbox{ divides } \bm{z}^{\bm{y}^+}.
         $$
         Then there exists $X_{ik}> 0$, so
         $$
         x_{ik} d_i \mbox{ divides } \bm{z}^{\bm{y}^+}.
         $$
         Similarly, there exists $T_{ik} < 0$ and
         $$
         t_{ik} \mbox{ divides } \bm{z}^{\bm{y}^-},
         $$
         so the initial term of $\bm{z}^{\bm{y}^+} - \bm{z}^{\bm{y}^-}$ is $\bm{z}^{\bm{y}^+}$ because $d_i > t_{ik}$. This initial term is divisible by $x_{ik} d_i$, which is the initial term of $x_{ik} d_i - t_{ik} b^{c_{ik}}$. Therefore
         $$
         {\rm in}_{<}(I_A)  = \langle {\rm in}_{<} ({\mathcal G}) \rangle,
         $$
         which proves that ${\mathcal G}$ is a Gr\"{o}bner basis of $I_A$ with respecto to the term order $<$. Clearly, it is reduced.
   \end{enumerate}
   Moreover, with respect to the term order $<_{\bm{c}}$,
         $$
         {\rm in}_{<_{\bm{c}}} (x_{ik} d_i - t_{ik} b^{c_{ik}} ) = x_{ik} d_i,
         $$
         because the weight of the first monomial is equal to $c_{ik} > 0$, and the weight of the second monomial is equal to zero. Similarly,
         $$
         {\rm in}_{<_{\bm{c}}}(x_{iq} t_{ip} - x_{ip} t_{iq} b^{c_{iq} - c_{ip}}) = x_{iq} t_{ip}, 1 \le q < p < k_i,
         $$
         because the weight of the first monomial is $c_{iq} \ge c_{ip}$, which is the weight of the second monomial. If $c_{iq} = c_{ip}$, the tie is broken with the lexicographical order $x_{iq} > x_{ip}$.

         \qed
 \end{proof}

 The above theorem gives a {\em reduced} Gr\"{o}bner basis with respect to the term order induced by the objective function of $(LR)$. On the contrary, \cite[Thm. 4]{tayur} only provides a Gr\"{o}bner basis with respect to a lexicographical order, and not with respect to the term order needed for the computation of the test set. Therefore, in order for that Gr\"{o}bner basis to be applied to solve their problem one more computational step is required whereas our construction gives directly the answer with its consequent saving.

 \section{Computational results} \label{sec:computational}
 The previous construction of the test set is used in our computational experiments. In order to gain some insights of its efficiency, if a program like \texttt{4ti2} (\cite{4ti2}) were used to do the computation of the test set, a simple configuration of 10 subsystems with 3 components would take more than 60 minutes. Therefore, it is very important to apply the result in Theorem \ref{thm:base} to be able to construct the test set.

 Our algorithm has been coded in {\sc Matlab} and run on a AMD Opteron 252 (2.6 GHz) with 5 GB RAM. For Table \ref{tabla01}, all the data in the test problems are randomly generated from uniform distributions, with $l_{ij} = 0, u_{ij} = 4$ and $r_{ij} \in [0.99, 0.998]$, as in \cite{ruan}. The linear cost function $\sum_{i=1}^n \sum_{j=1}^{k_i} c_{ij} x_{ij}$ has values $c_{ij} \in [10, 20]$.

 The number $n$ is the number of subsystems, and $k$ is the number of different components in each subsystem.

 \begin{table}[ht]
 \caption{$R_0 = 0.90, r_{ij} \in [0.99, 0.998]$}
 \label{tabla01}
 \begin{tabular}{ccrrrr} \hline\noalign{\smallskip} $n$ & $k$ && {\footnotesize  Nodes} & {\footnotesize Iter. R-S} &{\footnotesize  Avg. CPU time (s)} \\ \noalign{\smallskip}\hline\noalign{\smallskip} $10$ & $2$ & & 0 & $696$ & $0.0$ \\ $10$ & $3$ & & 0 & $5797$ & $0.0$ \\ $10$ & $5$ & & 0 & $26427$ & $0.0$ \\ $15$ & $2$ & &$0$& $15184$ & $0.0$  \\ $15$ & $3$ & & $0.4$ & $85103$ & $0.1$  \\ $20$ & $2$ && $7041$ &$294747$& $276.0$ \\ \noalign{\smallskip}\hline
  \end{tabular}
 \end{table}

 The average CPU time, and the average number of generated nodes during the algorithm has been obtained by running the program for 10 instances.

 The column ``Iter. R-S'' contains the number of iterations according to the results given in \cite[Table 1]{ruan}. Comparing our results with \cite{ruan}, not only the CPU time is improved, but also the effort measured by the number of processed nodes by the walk back procedure is less than the number of iterations in \cite{ruan}. We also point out that the iterations in \cite{ruan} compute two Lagrangian discrete relaxations and their corresponding solutions for the best value, each time. After that, to discard remaining boxes in their branch and bound tree, reliability of that solution is needed. In our method in each iteration we only compute a node by adding a vector, and then compute its reliability.


 In order to better illustrate the results, new tests have been done with the additional hypothesis that a greater reliability in a component implies a greater cost. Note that if there is no correlation between cost and reliability of a component (as in \cite{ruan}), then it is likely that certain components are not going to be used, because only more reliable components are going to be chosen regardless of their cost. Hence the dimensionality of the problem is artificially reduced. The results of this more realistic case appear in Table \ref{tabla02}.
 \begin{table}[ht]
 \caption{$R_0 = 0.90, r_{ij} \in [0.99, 0.998]$ ({\bf ordered})}
 \label{tabla02}
 \begin{tabular}{ccrrrr} \hline\noalign{\smallskip} $n$ & $k$ && {\footnotesize  Nodes} &  &{\footnotesize  Avg. CPU time (s)} \\ \noalign{\smallskip}\hline\noalign{\smallskip} $10$ & $2$ & & $0$ & & $0.0$ \\ $10$ & $3$ & & $0$ & & $0.0$ \\ $10$ & $5$ & & $0$ & & $0.0$ \\ $15$ & $2$ & &$45$& & $0.2$ \\ $15$ & $3$ & & $661$ && $4.4$  \\ $15$ & $4$ & & $28023$ & & $10571$ \\ $17$ & $2$ && $12578$ && $1355$ \\ \noalign{\smallskip}\hline
  \end{tabular}
 \end{table}
 It is clear the increasing computational effort showed by rows $n=15, k=3$ and $n=15, k=4$. From the above, we conclude that the computational experiments for this model should be done with this additional hypothesis of correlation between cost and reliability of each component.

 We also note that the algorithm is very sensitive to changes in the value of the reliability parameters $r_{ij}$. For example, for less reliable components, $r_{ij} \in [0.980, 0.990]$ and the same value, $R_0 = 0.90$, for the overall reliability, we have obtained the results in Table \ref{tabla03}.
 \begin{table}[ht]
 \caption{$R_0 = 0.90, {\mathbf r}_{ij} \in [\mathbf{0.98}, \mathbf{0.99}]$ ({\bf ordered})}
 \label{tabla03}
 \begin{tabular}{ccrrrr} \hline\noalign{\smallskip} $n$ & $k$ && {\footnotesize  Nodes} & &{\footnotesize  Avg. CPU time (s)} \\ \noalign{\smallskip}
  \hline \noalign{\smallskip}$6$ & $4$ & & $14$ & & $0$ \\ $6$ & $5$ & & $39$ & & $0.1$ \\ $7$ & $4$ & & $1186$ & & $7.4$ \\ $7$ & $5$ & &$5662$& & $140$ \\ $8$ & $4$ & & $46709$ && $7010$  \\ \noalign{\smallskip}\hline
  \end{tabular}
 \end{table}
 The reader may observe that the system sizes that can be solved are smaller. However, we have to point out that an exact solution has been found in all the examples. An interesting remark is that the elapsed time is significantly reduced if the algorithm is stopped with the first best point found in the walk back procedure. Obviously, this approach does not guarantee optimality but it gives very accurate approximations. From this observation, we think that a promising open field is the combination of this technique with heuristic methods to get a good approximation of the optimal solution.

 \section{Conclusion} \label{sec:conclusiones}
 We have presented in this paper an exact method for solving a nonlinear integer programming problem arising from the design of series-parallel reliability systems. The method is based on the construction of a test set of an integer linear problem through the theory of Gr\"{o}bner bases. We provide an explicit formula of the test set, avoiding the high cost of this computation. Computational tests show that this approach improves existing methods already applied for this problem.

 This paper deepens the challenge given in \cite{tayur} to yield efficient algorithms for problems in integer problems based on attractive bases.


\end{document}